\documentclass[11pt]{amsart}
 
\usepackage{hyperref, color}
\newtheorem{theorem}{Theorem}[section]
\newtheorem{lemma}[theorem]{Lemma}

\newtheorem{corollary}[theorem]{Corollary}

\theoremstyle{definition}

\theoremstyle{remark}
\newtheorem{remark}[theorem]{Remark}

\numberwithin{equation}{section}

\begin{document}
 
\title
[Index estimates for harmonic Gauss maps]
{Index estimates for harmonic Gauss maps}

\author[Carvalho]{A. de Carvalho}
\address{
 Departamento de Matemática,  Universidade Federal de Pernambuco, Recife, Brazil}
\email{alcides.junior@ufpe.br }

\author[Cavalcante]{M. Cavalcante$^*$}
\address{
 Instituto de Matemática, Universidade Federal de Alagoas,  
 Maceió, Brazil}
\email{marcos@pos.ufal.br}
\thanks{$^*$Corresponding author}

\author[Costa-Filho]{W. Costa-Filho}
\address{
 Campus Arapiraca, Universidade Federal de Alagoas, 
 Arapiraca, Brazil}
\email{wagner.filho@arapiraca.ufal.br}

\author[Oliveira]{D. de Oliveira}
\address{
Departamento de Ciências Exatas,  Universidade Estadual de Feira de Santana, Feira de Santana, Brazil}
\email{dfoliveira@uefs.br}


\date{\today}

\commby{}

\begin{abstract} 
Let $\Sigma$ denote a closed surface with constant mean curvature in $\mathbb{G}^3$, a 3-dimensional Lie group equipped with a bi-invariant metric. For such surfaces, there is a harmonic Gauss map which maps values to the unit sphere within the Lie algebra of $\mathbb{G}^3$.
We prove that the energy index of the Gauss map of $\Sigma$ is bounded below by its topological genus. We also obtain index estimates in the case of complete non-compact surfaces.  
\end{abstract}

\maketitle 


\section{Introduction}

Let $\Sigma$ be a closed, connected, orientable surface immersed in Euclidean space $\mathbb{R}^3$.
The classical Ruh-Vilms theorem \cite{Ruh,RV} states that $\Sigma$ has constant mean curvature (CMC) if and only if the Gauss map $N:\Sigma\to\mathbb{S}^2\subset\mathbb{R}^3$ is a harmonic map. Consequently, its energy index $Ind_{\mathcal{E}}(N)$ is well defined, and therefore, we can ask how it is related to the geometry and topology of the surface. 
This problem was addressed by Palmer in \cite{Palmer91}, where it was demonstrated that $Ind_{\mathcal{E}}(N)$ is bounded below by the largest integer $< 1+ g(\Sigma)/2$, where $g(\Sigma)$ represents the topological genus of $\Sigma$.

Later, by considering the canonical embedding $\mathbb{S}^3\subset \mathbb{R}^4$, Palmer \cite{Palmer93} obtained an estimate of the energy index of the Gauss map of constant mean curvature (CMC) surfaces in the unit sphere. In this context, the Gauss map is considered to take values in one factor of the Grassmannian of oriented 2-planes in $\mathbb{R}^4$, denoted by $G_{2,4} = \mathbb{S}^2\times \mathbb{S}^2$. In this setting, he demonstrated that $Ind_{\mathcal{E}}(N)\geq g(\Sigma)$.

An analogous problem concerning the index of the area functional for CMC surfaces was considered by the second and fourth authors in \cite{CO}, inspired by the study of minimal hypersurfaces (see \cite{Savo, ACS}). It has been proved that if $\Sigma$ is a closed CMC surface immersed in the Euclidean space $\mathbb R^3$, then $Ind(\Sigma)\geq g(\Sigma)/3$. Similarly, for a closed CMC surface immersed in the unit sphere $\mathbb S^3$, it holds that $Ind(\Sigma)\geq g(\Sigma)/4$. Examples of CMC surfaces of arbitrary genus were constructed in the Euclidean space by Wente \cite{wente} and Kapouleas \cite{Ka1, Ka2, Ka3}, and within the unit sphere by Butscher and Pacard \cite{B1, B2, BP}. Subsequently, the second and fourth authors together with Silva \cite{COR} proved lower bounds for $Ind(\Sigma)$ in terms of the genus for CMC surfaces immersed in 3-manifolds with the Killing property, which includes, in particular, Lie groups with a bi-invariant metric.

In this article, we consider the problem of estimating the energy index of the Gauss map of CMC surfaces immersed in a three-dimensional Lie group $\mathbb{G}^3$ endowed with a bi-invariant metric, as described below. It is worth noting that such groups can be obtained as quotients of the 3-sphere $\mathbb{S}^3$ or Euclidean space $\mathbb{R}^3$ (see \cite{EFFR}). For the reader’s convenience, we recall that the three-dimensional Lie groups endowed with a bi-invariant metric, as stated in \cite{EFFR}, are the Euclidean space $\mathbb{R}^{3}$ and the sphere $\mathbb{S}^3$, the projective space $\mathbb{RP}^3 \,(= \mathrm{SO}(3))$, and the products $\mathbb{T}^3 = \mathbb{S}^1 \times \mathbb{S}^1 \times \mathbb{S}^1$, $\mathbb{S}^1 \times \mathbb{S}^1 \times \mathbb{R}$, and $\mathbb{S}^1 \times \mathbb{R}^2$, all of which have constant sectional curvature with respect to the chosen bi-invariant metric. However, working with these quotients introduces complications that are not present in the classical settings. For instance, a closed surface in the quotient may correspond to a non-compact surface in the universal cover or may have a different genus, which can significantly affect the geometric analysis. This situation contrasts with Palmer's work \cite{Palmer91, Palmer93}, where the ambient space does not involve such quotients. In this paper, we tackle these challenges by extending Palmer's results to CMC surfaces in $\mathbb{G}^3$, providing energy index estimates for their Gauss maps.

Let $\Sigma\subset \mathbb{G}^3$ be a compact, oriented, immersed surface, and let $\eta$ denote a global unit normal vector field along $\Sigma$. Using the Lie group structure and denoting by $\mathbb{S}^2$ the unit sphere in the Lie algebra $\mathfrak{g} \cong T_e\mathbb{G}^3$, we define the \emph{generalized Gauss map} of $\Sigma$ as the map $\mathcal{N}:\Sigma \to \mathbb{S}^2$ given by
$$
\mathcal{N}(p) = d(L_{p^{-1}})_p(\eta(p)),
$$
where $L_p$ is the left translation in $\mathbb{G}^3$. It is important to note that, in the case where $\mathbb{G}^3 = \mathbb{S}^3$, this Gauss map differs from the one considered by Palmer in \cite{Palmer93}.

It turns out that a Ruh-Vilms type theorem applies in this scenario. More precisely, $\Sigma\subset \mathbb G^3$ has constant mean curvature if and only if $\mathcal N$ is a harmonic map, and thus, $Ind_{\mathcal E}(\mathcal N)$ is well defined. This result was established by Masal'tsev in \cite{Masaltsev} for the case of surfaces in the 3-sphere and by Espírito Santo et al. in \cite{EFFR} for hypersurfaces immersed in Lie groups equipped with a bi-invariant metric.

Our main theorem below establishes a generalization of Palmer's theorem to this setting. It reads as follows.

\begin{theorem}\label{maintheorem1} 
Let $\Sigma$ be a closed, orientable constant mean curvature $H$ surface immersed in a three-dimensional Lie group 
$\mathbb G^3$ endowed with a bi-invariant metric. If $\mathbb G^3$ is an Abelian Lie group, assume further that $H\neq 0$. Then,
$$
Ind_\mathcal E (\mathcal N)\geq g(\Sigma).
$$
\end{theorem}

\begin{remark}
In an oriented three–dimensional Riemannian manifold, an immersed surface is orientable if and only if it is two–sided.
Since our definition of constant mean curvature and of the generalized Gauss map requires the existence of a global unit normal 
vector field, all surfaces considered in this paper are automatically orientable.
\end{remark}

The condition $H \neq 0$ for Abelian Lie groups is essential, as it naturally arises in the proof and is exemplified by the case of the totally geodesic torus $\mathbb T^2$ embedded in $\mathbb T^3$, where the harmonic Gauss map has an index of zero, yet $g(\mathbb T^2)=1$.

Now, recall that Chern \cite{Chern} proved that any immersed constant mean curvature sphere in a 3-dimensional space form must be a geodesic sphere. From this result, we derive the following direct consequence of Theorem \ref{maintheorem1}:

\begin{corollary}
Let $\Sigma$ be a closed surface with constant mean curvature, immersed in $\mathbb{S}^3$. If the generalized Gauss map $\mathcal{N}$ is stable, then $\Sigma$ must be a geodesic sphere in $\mathbb{S}^3$.
\end{corollary}

We also consider the case of complete, non-compact surfaces. Since harmonicity is a local property, the generalized Gauss map of complete CMC surfaces in $\mathbb{G}^3$ remains harmonic, and the same considerations apply. However, under this condition, the index $Ind_{\mathcal{E}}(\mathcal{N})$ could potentially be infinite. This problem was considered by Palmer \cite{Palmer93} in the Euclidean space, and in this case, the space of $L^2$-harmonic 1-forms, $\mathcal H^1(\Sigma)$, plays the role of the genus. In fact, if $g(\Sigma)$ is the number of handles of $\Sigma$, then $\dim \mathcal{H}^1(\Sigma)\geq g(\Sigma)$. Of course, it is convenient to use the same notation as the genus, because these numbers coincide in the closed case.

We have the following result:
\begin{theorem}\label{maintheorem2} 
Let $\Sigma$ be a complete non-compact, orientable constant mean curvature $H$ surface immersed in a three-dimensional Lie group $\mathbb{G}^3$ endowed with a bi-invariant metric. If $\mathbb{G}^3$ is an Abelian Lie group, assume further that $H \neq 0$. Then,
$$
\text{Ind}_{\mathcal{E}}(N) \geq \dim \mathcal{H}^1(\Sigma)/2.
$$
\end{theorem}

We recall that the Ruh-Vilms theorem admits extensions for suitable Gauss maps in other ambient spaces. For example, extensions apply to the sphere $\mathbb{S}^{7}$ or $\mathbb{CP}^{3}$, as proved by Bittencourt et al. in \cite{bittencourt}; as well as to homogeneous manifolds with invariant Riemannian metrics, as explored by Bittencourt and Ripoll \cite{BR}, and to symmetric spaces, as demonstrated by Ramos and Ripoll \cite{RamosRipoll}. We also refer to \cite{Daniel, DFM, FM} for the case of harmonic Gauss maps from CMC surfaces in some homogeneous spaces into the hyperbolic plane $\mathbb H^2$. Moreover, suitable Gauss maps play a key role in the works of Folha and Peñafiel \cite{folha} and Gálvez and Mira \cite{GalvezMira}.

This leads to the natural question: does a version of Palmer's theorem hold in these varied contexts?

 \subsection*{Acknowledgments}
We would like to thank the anonymous referee for a careful reading of the manuscript and for several valuable comments that helped to improve the presentation.

The authors are grateful to Cezar Oniciuc and Dorel Fetcu for their helpful conversations and valuable comments on this work. The first author was partially supported by FAPEAL (E:60030.0000000323/ 2023). The second author  acknowledges support from the ICTP through the Associates Programme (2022-2027), from CNPq (Grants 405468/2021-0 and 11136/2023-0. This work was also carried out within the CAPES--Math AmSud project \emph{New Trends in Geometric Analysis} (CAPES, Grant 88887.985521/2024-00).


\section{The Gauss map of hypersurfaces in Lie groups}
Let $\mathbb{G}$ denote a $(n+1)$-dimensional Lie group endowed with a left-invariant Riemannian metric and with Lie algebra $\mathfrak{g}$.

Given a two–sided (hence orientable) hypersurface $\Sigma$ immersed in $\mathbb{G}$, 
let $\eta$ be a globally defined unit normal vector field along $\Sigma$.
In particular, throughout the paper we only consider two–sided immersions, so that the generalized Gauss map 
$\mathcal N : \Sigma \to \mathbb S^2$ is globally well defined.

In this setting, it is natural to use the group structure to define the map $\mathcal{N} : \Sigma \to \mathbb{S}^{n}$ which translates the unit normal vector field of $\Sigma$ to the identity of $\mathbb{G}$. More precisely,
\begin{equation*}
    \mathcal{N}(p) = d(L_{p}^{-1})_{p}(\eta(p)),\label{DefN}
\end{equation*}
where $L_p$ denotes the left translation by $p$ in $\mathbb{G}$.

This map $\mathcal{N}$ serves as a natural extension of the usual Gauss map for hypersurfaces in Euclidean spaces and has been studied by Ripoll in \cite{ripoll91}. See also \cite{EFFR} and \cite{GalvezMira}.

As mentioned in the introduction, it was demonstrated by Espírito Santo et al. in \cite{EFFR} that if the metric in $\mathbb G$ is bi-invariant, then a hypersurface $\Sigma \subset \mathbb G$ has constant mean curvature if and only if its Gauss map $\mathcal{N}$ is a harmonic map. This result extends the classical Ruh-Vilms theorem to the context of hypersurfaces in Lie groups equipped with a bi-invariant metric.
 
In order to study the extrinsic geometry of $\Sigma\subset \mathbb{G}$, it is essential to determine the differential of the Gauss map $\mathcal{N}$. As shown by Ripoll in \cite[Proposition~3]{ripoll91}, for every $p\in\Sigma$ one has
\begin{equation}\label{dNX}
    dL_p \circ d\mathcal{N}_p = -(A_p + \mathcal{A}_p),
\end{equation}
where $A_p$ is the classical shape operator of the immersion and $\mathcal{A}_p$ denotes the \emph{invariant shape operator}. The latter is defined as a section in the bundle $Hom(T\Sigma, T\Sigma)$ through the expression
$$
\mathcal{A}_p(X) = (\nabla_X \tilde{\eta_p})(p),
$$
with $\tilde{\eta_p}$ being the left invariant extension of the unit normal vector field $\eta$ such that $\eta(p) = \tilde{\eta}_p(p)$ at each point $p$ on $\Sigma$.
 Since the Levi--Civita connection is tensorial in the first argument, the value of $\mathcal{A}_p(X)$ is determined by $X(p)$ and the left invariant extension $\tilde{\eta_p}$ of the unit normal vector field $\eta$. 

To make this dependence explicit, fix $p\in\Sigma$ and choose an orthonormal basis $\{X_1,\ldots,X_{n+1}\}\subset T_p\mathbb{G}$ such that $X_{n+1}=\eta(p)$. Let $W_1,\ldots,W_{n+1}$ be the corresponding left-invariant vector fields obtained by left-translation, and denote by $Z_j = W_j(e)$ the induced orthonormal basis of the Lie algebra. Consider the left-invariant extension $\widetilde{\eta_p}$ of the normal direction at $p$, i.e., the unique left-invariant vector field satisfying $\widetilde{\eta_p}(p)=\eta(p)$. In terms of the frame $\{Z_j\}$, this field can be written as
$$
\widetilde{\eta_p}=\sum_{j=1}^{n+1} a_j\, W_j,
$$
where the coefficients satisfy $a_j(p)=0$ for $1\le j\le n$ and $a_{n+1}(p)=1$. 

In this framework, and using the fact that $\widetilde{\eta}_p=W_{n+1}$, the invariant shape operator at $p$ is determined by the matrices
$$
\alpha_B=\big( \langle \nabla_{W_i}W_{n+1},\, W_j\rangle \big)_{1\le i,j\le n},
$$
as stated in \cite[Remark~4]{ripoll91}.

For any isometric automorphism $F$ of $\mathfrak{g}$, we have that 
\begin{equation}\label{invariant}
    \alpha_B = \alpha_{F(B)}.
\end{equation}

Consequently, if $O_{n+1}$ denotes the set of orthonormal $(n+1)$-frames in
$\mathfrak{g}$, the collection of matrices $\alpha_B$ can be parametrized by
the quotient space $O_{n+1}/I(\mathbb{G})$, where $I(\mathbb{G})$ denotes the
group of isometric automorphisms of $\mathbb{G}$. In particular, we obtain the
following special cases; the first two are due to \cite{ripoll91}.

\begin{itemize}
 \item[1.] {\it The commutative case}. When $\mathbb{G}$ is commutative, then the invariant shape operator $\mathcal{A}$ of $\Sigma$ is identically zero, and $\mathcal{N}$ is the
usual Gauss map of $\Sigma.$ In fact, since $\mathbb{G}$ admits a bi-invariant metric, we have that $\nabla_XY =\dfrac{1}{2} [X, Y ] ,$ whereas due to its commutativity, we arrive at $\nabla_XY =0.$

\item[2.]{\it The sphere  $\mathbb{S}^3$.} We consider in $\mathbb{S}^3$ a bi-invariant metric. It is known that $Ad_{\mathbb{S}^3} = \mathsf O(3),$ that is, the image of adjoint representation of $\mathbb{S}^3$ is the orthogonal group $\mathsf O(3),$ which acts transitively on $O_3.$ Therefore, by (\ref{invariant}), $\mathcal{A}$ is constant.
Computations using the quaternionic model for $\mathbb{S}^3$ yields
$$
\alpha_B=\begin{pmatrix}
0 & 1 \\
-1 & 0
\end{pmatrix}.
$$
\item[3.]{\it The projective space $\mathbb{R}P^3$.} We consider a bi-invariant metric on $\mathsf{SO}(3)$. It is known that the image of the adjoint representation of $\mathsf{SO}(3)$ is $\mathsf{SO}(3)$ itself. In other words, under the adjoint representation, $\mathsf{SO}(3)$ maps to the rotation group $\mathsf{SO}(3)$, which has two orbits when acting on $O_3$. Therefore, using (\ref{invariant}) again, we get
$$
\alpha_B=\frac{1}{2}\begin{pmatrix}
0 & 1 \\
-1 & 0
\end{pmatrix}
\quad \mbox{or}\quad 
\alpha_B=\frac{1}{2}\begin{pmatrix}
0 & -1 \\
1 & 0
\end{pmatrix}.
$$
\end{itemize}

Recalling the classification of 3-dimensional Lie groups with a bi-invariant metric (see \cite{EFFR}), we can summarize the above discussion as follows. 

\begin{lemma}\label{lemma:normA}
Let $\Sigma$ be an orientable surface immersed in a three-dimensional Lie group $\mathbb{G}^3$ equipped with a bi-invariant metric. Then, the invariant shape operator $\mathcal{A}$ is either identically zero or takes one of the following forms on orthonormal basis $B$ of left invariant vector fields $B$:
$$
\alpha_B=\begin{pmatrix}
0 & 1 \\
-1 & 0
\end{pmatrix},\quad\alpha_B=\frac{1}{2}\begin{pmatrix}
0 & 1 \\
-1 & 0
\end{pmatrix}, \quad \text{or} \quad \alpha_B=\frac{1}{2}\begin{pmatrix}
0 & -1 \\
1 & 0
\end{pmatrix}.
$$
\end{lemma}

To conclude this section, we need to explore one additional property of the invariant shape operator $\mathcal{A}$.

\begin{lemma}\label{alphaA}
Let $\Sigma^n$ be an oriented hypersurface immersed in a Lie group $\mathbb{G}^{n+1}$ endowed with a bi-invariant metric. Then, the invariant shape operator $\alpha_B$, corresponding to an orthonormal basis $B$ of left invariant vector fields, is an anti-symmetric matrix. In particular,
$$
\langle \mathcal{A}, A \rangle = 0,
$$
meaning that the Hilbert-Schmidt inner product between the invariant shape operator $\mathcal{A}$ and the shape operator $A$ is zero.
\end{lemma}
\begin{proof} Let $B = \{W_1, \ldots, W_{n+1}\}$ be an orthonormal basis of left invariant vector fields and
let  $\alpha_B = (a_{ij})$ be the matrix representation of the invariant shape operator, for $1 \leq i, j \leq n$. The elements $a_{ij}$ are defined by $a_{ij} = \langle \nabla_{W_i}W_{n+1}, W_j \rangle$.

Given that $\mathbb{G}$ is equipped with a bi-invariant metric, the covariant derivative relates to the Lie bracket as follows:
\begin{equation*}
\nabla_{X}Y = \frac{1}{2} [X, Y],
\end{equation*}
for any left invariant vector fields $X$ and $Y$ in the Lie algebra $\mathfrak{g}$ of $\mathbb{G}$.

Hence, by applying the properties of the connection and the Lie brackets, we derive:
\begin{align*}
a_{ij} &= \langle \nabla_{W_i}W_{n+1}, W_j \rangle \\
       &= -\langle W_{n+1}, \nabla_{W_i}W_j \rangle \\
       &= -\frac{1}{2} \langle W_{n+1}, [W_i, W_j] \rangle \\
       &= \frac{1}{2} \langle W_{n+1}, [W_j, W_i] \rangle \\
       &= -a_{ji}.
\end{align*}

This demonstrates that the matrix $\alpha_B$ is anti-symmetric.
\end{proof}

\section{The energy index of harmonic maps}\label{Harmonic}

Let $\Sigma$ be a closed 2-dimensional Riemannian surface and let
$\mathcal{N}: \Sigma \rightarrow \mathbb{S}^2$ be a harmonic map from $\Sigma$
to the unit 2-sphere. This means that $\mathcal{N}$ is a critical point of the
energy functional
$$
\mathcal{E}(\mathcal{N}) = \frac{1}{2} \int_\Sigma e(\mathcal{N}) \, d\Sigma,
$$
where $e(\mathcal{N}) = \sum_i \langle d\mathcal{N} e_i, d\mathcal{N} e_i
\rangle$ is the energy density computed in a local orthonormal frame
$\{e_1, e_2\}$ on $\Sigma$.

For a section $\xi$ of the pull-back bundle $\mathcal{N}^{\ast}T\mathbb{S}^2$, 
the second variation of the energy $\mathcal{E}$ at a harmonic map $\mathcal{N}$, 
as proved by Mazet \cite{Mazet} and Smith \cite{Smith}, is given by
$$
D^2_{\xi}\mathcal{E} = \int_\Sigma \langle \xi, J_{\mathcal{N}} \xi \rangle \, d\Sigma.
$$

Here, $J_{\mathcal{N}}$ denotes the Jacobi operator. In the special case where 
the target manifold is the unit sphere, it is given by
$$
J_{\mathcal{N}} \xi = -\tilde{\nabla}^{\ast}\tilde{\nabla} \xi 
 - \sum_{i=1}^2  R_{\mathbb{S}^2}(\xi, d\mathcal{N}e_i)\,d\mathcal{N}e_i,
$$
where $\tilde{\nabla}$ denotes the connection on the bundle 
$\mathcal{N}^{\ast}T\mathbb{S}^2$, $\tilde{\nabla}^{\ast}$ is its formal adjoint, 
and $R_{\mathbb{S}^2}$ is the curvature tensor of $\mathbb{S}^2$ (see \cite{Leung}).

Since $\Sigma$ is closed, $J_{\mathcal N}$ is a self-adjoint elliptic operator
with discrete spectrum. Thus there exists a sequence of real eigenvalues
$$
\lambda_1 \le \lambda_2 \le \lambda_3 \le \cdots \to +\infty,
$$
each with finite multiplicity, and a corresponding $L^2$–orthonormal basis of
eigenvectors. If $\xi$ is an eigenvector with eigenvalue $\lambda$, that is,
$J_{\mathcal N}\xi = \lambda \xi$, then
$$
D^2_\xi \mathcal E = \int_\Sigma \langle \xi, J_{\mathcal N}\xi \rangle\, d\Sigma
 = \lambda \int_\Sigma \|\xi\|^2\, d\Sigma.
$$

Moreover, the eigenvalues admit the usual min–max characterization (see, for
instance, \cite[p.~566]{Palmer91}): for each $k \ge 1$ one has
$$
\lambda_k 
= \sup_{\mathcal V^{k-1}} \ \inf_{\substack{\xi \in (\mathcal V^{k-1})^\perp \\ \xi \neq 0}}
\frac{D^2_\xi\mathcal E }{\displaystyle\int_\Sigma \|\xi\|^2\, d\Sigma},
$$
where the supremum is taken over all subspaces $\mathcal V^{k-1}$ of
$\Gamma(\mathcal{N}^\ast T\mathbb S^2)$ of dimension $k-1$, and the orthogonal
complement is taken with respect to the $L^2$ inner product.

The \emph{energy index} of $\mathcal N$, denoted by $\mathrm{Ind}_{\mathcal E}(\mathcal N)$,
is defined as the number of negative eigenvalues of $J_{\mathcal N}$ counted
with multiplicities. This is a significant geometric invariant that reflects
the number of linearly independent variations of the harmonic map $\mathcal N$
that decrease the energy. A harmonic map $\mathcal{N}$ is called
\emph{stable} if $\mathrm{Ind}_{\mathcal{E}}(\mathcal{N}) = 0$.

\subsection{Harmonic vector fields as test variations}\label{testvariations}

In the particular case of the harmonic Gauss map $\mathcal{N}$, as defined in
\eqref{DefN}, we observe that there exists a bundle isometry between $T\Sigma$
and $\mathcal{N}^{\ast}T\mathbb{S}^2$. Consequently, vector fields on $\Sigma$
can be used as test variations for the second derivative of the energy
functional. More precisely, for any vector field $\xi \in T\Sigma$ there
corresponds a section $\bar{\xi} \in \Gamma(\mathcal{N}^{\ast}T\mathbb{S}^2)$,
defined by
$$
\bar{\xi}(p) = (dL_{p^{-1}})_p \xi(p),
$$
where $T_p\Sigma$ is regarded as a subspace of $T_p\mathbb{G}^3$.

In this situation, the second variation can be written as
\begin{align*} 
D^2_\xi \mathcal E &=   
\int_\Sigma \Big\langle -\nabla^{\ast} \nabla \xi 
-\sum_i R_{\mathbb S^2}(\xi, d\mathcal{N}e_i)d\mathcal{N}e_i, \xi \Big\rangle\, d\Sigma,
\end{align*}
where $\nabla$ denotes the Riemannian connection on $T\Sigma$.

If $K_\Sigma$ stands for the Gauss curvature of $\Sigma$, we then use the
Bochner formula $\nabla^*\nabla \xi = -\Delta \xi + K_\Sigma \xi$ and the Gauss
equation to derive
\begin{align}\label{formulaE}
D^2_\xi \mathcal E =
\int_\Sigma \langle \Delta \xi - K_\Sigma \xi, \xi \rangle\, d\Sigma
 - \int_\Sigma \Big(\sum_i \|\xi\|^2 \|d\mathcal{N}e_i\|^2 - \sum_i \langle \xi, d\mathcal{N}e_i \rangle^2 \Big) d\Sigma.
\end{align}

Set
$$
F(\xi) = \sum_i \|\xi\|^2 \|d\mathcal{N}e_i\|^2 - \sum_i \langle \xi, d\mathcal{N}e_i \rangle^2.
$$
To compute $F(\xi)$ we use the identity \eqref{dNX} and the fact that $dL_p$
is an isometry. Let $\{e_i\}$ be a local orthonormal tangent frame at a point
$p\in\Sigma$. Since
$$
dL_p\circ d\mathcal N_p = -(A_p+\mathcal A_p),
$$
we have, for every tangent vector $v$,
$$
\|d\mathcal N(v)\| = \|(A+\mathcal A)v\|,
\qquad
\langle \xi, d\mathcal N(v)\rangle = \langle \xi,(A+\mathcal A)v\rangle.
$$

Hence
\begin{align*}
\sum_i \|d\mathcal N e_i\|^2
&= \sum_i \langle (A+\mathcal A)e_i,(A+\mathcal A)e_i\rangle \\
&= \|A\|^2 + \|\mathcal A\|^2 + 2\langle A,\mathcal A\rangle.
\end{align*}
By Lemma~\ref{alphaA} we have $\langle A,\mathcal A\rangle=0$, and by
From Lemma~\ref{lemma:normA} we know that $\|\mathcal A\|^2 = 2c$, where $c=0$ if $\mathbb{G}$ is Abelian, $c=1$ when $\mathbb{G}$ is the sphere $\mathbb{S}^3$, and $c=\tfrac14$ in the projective case $\mathbb{R}P^3$. Therefore,

$$
\sum_i \|d\mathcal N e_i\|^2 = \|A\|^2 + 2c.
$$

Similarly,
\begin{align*}
\sum_i \langle \xi, d\mathcal N e_i\rangle^2
&= \sum_i \langle \xi,(A+\mathcal A)e_i\rangle^2
 = \|(A+\mathcal A)^t\xi\|^2 \\
&= \|A\xi\|^2 + \|\mathcal A\xi\|^2 - 2\langle A\xi,\mathcal A\xi\rangle,
\end{align*}
where we used that $A$ is symmetric and $\mathcal A$ is skew-symmetric, so
$(A+\mathcal A)^t = A-\mathcal A$. By Lemma~\ref{lemma:normA} we also have
$\|\mathcal A\xi\|^2 = c\|\xi\|^2$, and thus
$$
\sum_i \langle \xi, d\mathcal N e_i\rangle^2
= \|A\xi\|^2 + c\|\xi\|^2 - 2\langle A\xi,\mathcal A\xi\rangle.
$$

Putting these identities together, we obtain
\begin{align}\label{F}
F(\xi)
&= \|\xi\|^2(\|A\|^2+2c) - \big(\|A\xi\|^2 + c\|\xi\|^2 - 2\langle A\xi,\mathcal A\xi\rangle\big) \\
\nonumber
&= (\|A\|^2 + c) \|\xi\|^2 - \|A\xi\|^2  
+ 2 \langle A\xi, \mathcal{A} \xi \rangle.
\end{align}

The Gauss equation for the immersion of $\Sigma$ in $\mathbb{G}^3$ implies
$$
\|A\|^2 = 4H^2 - 2K_{ext} = 4H^2 - 2K_\Sigma + 2c,
$$
where $H$ is the mean curvature of $\Sigma$ and $K_{ext}$ is the product of
the principal curvatures of $\Sigma$. A direct computation yields
$$
\|A\xi\|^2 = 2H \langle A\xi, \xi \rangle - K_\Sigma \|\xi\|^2 + c \|\xi\|^2.
$$

Inserting these identities into \eqref{F}, we obtain
$$
F(\xi) = -K_\Sigma \|\xi\|^2
+ (4H^2 + c) \|\xi\|^2 - 2H \langle A\xi, \xi \rangle
+ \langle A\xi, \mathcal{A}\xi \rangle.
$$

Finally, substituting this expression for $F(\xi)$ into \eqref{formulaE}, we get
\begin{align}\label{mainformulacompletecase}
D^2_\xi \mathcal E =    \int_\Sigma \big(\langle \xi,\Delta \xi\rangle - (4H^2 +c) \|\xi\|^2\big)\,d\Sigma 
 + \int_\Sigma \big(2H \langle A\xi, \xi\rangle - \langle A \xi, \mathcal{A} \xi\rangle\big)\,d\Sigma.
\end{align}

Following the ideas of Palmer \cite{Palmer91, Palmer93}, we will use harmonic
vector fields on $\Sigma$ as test vector fields in this index form. In this
case, the first integral in \eqref{mainformulacompletecase} vanishes and we
obtain
\begin{align}\label{mainformula}
D^2_\xi \mathcal E =  - (4H^2 +c) \int_\Sigma\|\xi\|^2\,d\Sigma  
 +\int_\Sigma \big(2H   \langle A\xi, \xi\rangle  -  \langle A \xi, \mathcal{A} \xi\rangle\big)d\Sigma.
\end{align}

\section{Proof of Theorem \ref{maintheorem1}}

\begin{proof}
Let us denote by $\mathcal{H}(\Sigma)$ the space of harmonic vector fields on $\Sigma$,
and let $J$ be the complex structure on $\Sigma$. If $\xi \in \mathcal{H}(\Sigma)$, then
$\xi^\perp = J\xi$ is a harmonic vector field pointwise orthogonal to $\xi$, and
$\|\xi^\perp\|^2 = \|\xi\|^2$. In particular, we have
$$
\langle A\xi, \xi\rangle + \langle A\xi^\perp, \xi^\perp\rangle = 2H\|\xi\|^2,
$$
and by Lemma \ref{alphaA},
$$
\langle A\xi, \mathcal{A} \xi\rangle + \langle A\xi^\perp, \mathcal{A} \xi^\perp\rangle = 0.
$$

Thus,
\begin{equation}\label{eq:neg}
D^2_\xi\mathcal E+D^2_{J\xi}\mathcal E
=-(4H^2+2c)\int_\Sigma\|\xi\|^2<0,
\end{equation}
indicating that $D^2_\xi \mathcal{E}$ or $D^2_{\xi^\perp} \mathcal{E}$ is negative. Note that if $\mathbb G^3$ is Abelian, that is, if
$c = 0$, we need to assume $H \neq 0$.

Let $k=\operatorname{Ind}_{\mathcal E}(\mathcal N)$ and let $E_-$ be the negative eigenspace of $D^2\mathcal E$ (dimension $k$). Consider
$$
T:\mathcal H(\Sigma)\to\mathbb R^{2k},\qquad
T(\xi)=\Big(\int_\Sigma\langle\xi,X_i\rangle\,d\Sigma,\int_\Sigma\langle\xi^\perp,X_i\rangle\,d\Sigma\Big)_{i=1}^k,
$$
where  $\{X_i\}_{i=1}^k$ is an $L^2$-orthonormal basis of $E_-$. If $2g(\Sigma)=\dim\mathcal H(\Sigma)>2k$, there exists $0\neq \xi\in\ker T$, so $\xi$ and $\xi^\perp$ are $L^2$-orthogonal to $E_-$ and hence $D^2_\xi\mathcal E\ge0$ and $D^2_{\xi^\perp}\mathcal E\ge0$, contradicting \eqref{eq:neg}. Therefore $2g(\Sigma)\le 2k$ and
$$
\operatorname{Ind}_{\mathcal E}(\mathcal N)\ge g(\Sigma).
$$
\end{proof}

\section{Index bounds for complete surfaces}

In this section, we consider the case of complete noncompact surfaces. As mentioned in the Introduction, if $\Sigma$ is a complete CMC surface immersed in $\mathbb{G}^3$, then the generalized Gauss map $\mathcal{N}:\Sigma\to \mathbb{S}^2$ is harmonic. Indeed, harmonicity is a local property. However, in the case of noncompact surfaces, we must consider compactly supported variations. In particular, considering the identification we used in Section \ref{testvariations}, the formula (\ref{mainformulacompletecase}) for the second variation of the energy functional holds for any $\xi \in T\Sigma$ with compact support, and the index $Ind_\mathcal{E}(\mathcal{N})$ is the maximal dimension of such sections where $D^2 \mathcal{E}$ is negatively defined.

A challenge in the noncompact case is that harmonic vector fields may not have compact support. To address this problem, it is natural to use $L^2$-harmonic vector fields. We will explore the correspondence between vector fields $\xi \in T\Sigma$ and 1-forms $\omega \in \Omega^1(\Sigma)$ given by the musical isomorphism. Namely, $\xi = \omega^\sharp$ and $\omega = \xi^\flat$, and $\xi^\perp=(\star \omega)^\sharp$, where $\star$ is the Hodge star operator.  In this context, the space of $L^2$-harmonic 1-forms, denoted by $\mathcal{H}^1(\Sigma)$, serves as a natural choice to capture the topological properties of $\Sigma$, effectively linking the geometry of the surface with its topological structure. In fact, if $\Sigma$ is complete, then $\mathcal{H}^1(\Sigma)$ is isomorphic to the first-reduced $L^2$-cohomology group of $\Sigma$ (see \cite[Corollary 1.6]{Carron01}), and in particular, $\dim \mathcal{H}^1(\Sigma)$ is a topological invariant. 
Moreover, if $\Sigma$ is a complete Riemannian surface with $g(\Sigma)$ handles, then $\dim \mathcal{H}^1(\Sigma) \geq g(\Sigma)$ (see \cite[Section 1.2.2]{Carron01}).

\smallskip

Now, inspired by Palmer \cite{Palmer93}, we need the following cut-off formula.

\begin{lemma}\label{lemma1}
Let $\xi \in T\Sigma$ be a smooth vector field and let $\varphi \in \mathcal{C}^{\infty}_c(\Sigma)$ be a compactly supported smooth function. Then, the second variation of energy in the direction of $ \varphi\xi \in T\Sigma $ is given by:
   \begin{align*}
   D^2_{\varphi \xi} \mathcal{E} =& \int_\Sigma \|\nabla\varphi\|^2 \|\xi\|^2\,d\Sigma + \int_\Sigma \varphi^2\langle \xi, \Delta \xi\rangle\,d\Sigma 
   - (4H^2 + c) \int_\Sigma \varphi^2 \|\xi\|^2\,d\Sigma  \\
   &+\int_\Sigma \varphi^2 (2H \langle A\xi, \xi\rangle - \langle A \xi, \mathcal{A} \xi\rangle)\,d\Sigma.
   \end{align*}
\end{lemma}

\begin{proof}
A direct computation shows that $\Delta (\varphi \xi) = \varphi \Delta \xi + 2 \nabla_{\nabla \varphi} \xi + (\Delta \varphi) \xi$.

Integrating this formula by parts and plugging it into formula \eqref{mainformulacompletecase}, we get the result.
\end{proof}

By applying this result to $\xi$ and $\xi^{\perp}$ and adding the results, we obtain:

\begin{lemma}\label{lemma2}
If $\xi$ is a harmonic vector field and $\varphi \in \mathcal{C}^{\infty}_c(\Sigma)$, then,
$$ 
   D^2_{\varphi \xi} \mathcal{E} + D^2_{\varphi \xi^{\perp}} \mathcal{E} = - (4H^2 + 2c) \int_\Sigma \varphi^2 \|\xi\|^2\,d\Sigma + 2 \int_\Sigma \|\nabla\varphi\|^2 \|\xi\|^2\,d\Sigma.
$$
\end{lemma}

In particular, choosing a standard cut-off function $\varphi(r)$ depending on the distance $r$ to a fixed point, we obtain
$$
D^2_{\varphi \xi} \mathcal{E} + D^2_{\varphi \xi^{\perp}} \mathcal{E} < 0.
$$

Now, recalling that the space of $L^2$-harmonic vector fields is isomorphic to $\mathcal{H}^1(\Sigma)$, we can proceed as in the proof of Theorem \ref{maintheorem1} to prove our result for complete surfaces:

\begin{theorem}
Let $\Sigma$ be a complete non-compact constant mean curvature $H$ surface immersed in a three-dimensional Lie group $\mathbb{G}^3$ endowed with a bi-invariant metric. If $\mathbb{G}^3$ is an Abelian Lie group,  assume further that $H \neq 0$.
Then,
$$
\text{Ind}_{\mathcal{E}}(N) \geq \dim \mathcal{H}^1(\Sigma)/2.
$$
\end{theorem}

\section{Energy index and area index}

In this section, we discuss the relationship between the energy index of the Gauss map and the area index of the surface. This connection provides a deeper understanding of the stability properties of CMC surfaces in the context of harmonic maps and their variational framework.

The stability of a closed CMC surface $\Sigma \subset \mathbb{G}^3$ is determined by the second variation of the area functional. Given $\phi \in C^\infty(\Sigma)$, a smooth function on $\Sigma$, the second variation of the area functional is given by:
$$
D^2 _\phi A = \int_\Sigma \left( \|\nabla \phi\|^2 - \big( \|A\|^2 + \text{Ric}(\eta,\eta) \big) \phi^2 \right) d\Sigma,
$$
where $\text{Ric}(\eta,\eta)$ is the Ricci curvature of $\mathbb{G}^3$ in the direction of the unit normal $\eta$ to $\Sigma$. We say that $\Sigma$ is stable if $D^2 _\phi A  \geq 0$ for all smooth functions $\phi$ on $\Sigma$.

On the other hand, as shown in Section \ref{Harmonic}, the second variation of the energy functional of the Gauss map $\mathcal{N} : \Sigma \to \mathbb{S}^2$ is given by:
$$
D^2_\xi  \mathcal{E} = \int_\Sigma \left( \|\nabla \xi\|^2 - F(\xi) \right) d\Sigma,
$$
where $F(\xi)$ depends on the geometric properties of $\Sigma$, as described in Equation \eqref{F}. 

Following Palmer’s work~\cite{Palmer91}, we compare the stability of the area functional with that of the energy functional in the case where $\mathbb{G}^3$ is an Abelian Lie group. Using Kato's inequality,
$$
\|\nabla \|\xi\|\|^2 \leq \|\nabla \xi\|^2,
$$
and if $\mathbb{G}^3$ is Abelian, $F(\xi) \geq \|A\|^2 \|\xi\|^2$. Substituting these into the expression for $D^2_\xi \mathcal{E}$, we find:
$$
D^2_ \xi \mathcal{E} \geq \int_\Sigma \left( \|\nabla \|\xi\|\|^2 - \|A\|^2 \|\xi\|^2 \right) d\Sigma = D^2_{\|\xi\|}A.
$$

This inequality establishes that if the area functional is stable for all variations, then the Gauss map $\mathcal{N}$ is also stable. 

To conclude, we recall a related result in the context of minimal immersions. Ejiri and Micallef~\cite{Ejiri_Micallef} studied minimal immersions of closed Riemannian surfaces, which are critical points of both the area and the Dirichlet energy functionals, and showed that the Morse index of the energy functional $ i_E $ and the Morse index of the area functional $ i_A $ satisfy
$$
i_E \leq i_A \leq i_E + r(g, b),
$$
where $ r(g, b) $ depends on the genus $ g $ and the number $ b $ of branch points of the immersion. This inequality was recently extended to the setting of constant mean curvature surfaces by Seemungal and Sharp~\cite{SeemungalSharp}. 

In both cases, the energy functional corresponds to the immersion itself. In contrast, our analysis concerns the energy of the harmonic Gauss map of a CMC immersion. Whether a similar index inequality holds in this context remains an open and intriguing question.

\bibliographystyle{amsplain}
\bibliography{referencias}

\providecommand{\bysame}{\leavevmode\hbox to3em{\hrulefill}\thinspace}
\providecommand{\MR}{\relax\ifhmode\unskip\space\fi MR }
\providecommand{\MRhref}[2]{%
  \href{http://www.ams.org/mathscinet-getitem?mr=#1}{#2}
}
\providecommand{\href}[2]{#2}
\begin{thebibliography}{10}

\bibitem{ACS}
Lucas Ambrozio, Alessandro Carlotto, and Ben Sharp, \emph{Comparing the {M}orse
  index and the first {B}etti number of minimal hypersurfaces}, J. Differential
  Geom. \textbf{108} (2018), no.~3, 379--410. \MR{3770846}

\bibitem{bittencourt}
Fidelis Bittencourt, Pedro Fusieger, Eduardo~R. Longa, and Jaime Ripoll,
  \emph{Gauss map and the topology of constant mean curvature hypersurfaces of
  {{\(\mathbb{S}^7\)}} and {{\(\mathbb{CP}^3 \)}}}, Manuscr. Math. \textbf{163}
  (2020), no.~1-2, 279--290.

\bibitem{BR}
Fidelis Bittencourt and Jaime Ripoll, \emph{Gauss map harmonicity and mean
  curvature of a hypersurface in a homogeneous manifold}, Pac. J. Math.
  \textbf{224} (2006), no.~1, 45--63.

\bibitem{B1}
Adrian Butscher, \emph{Constant mean curvature hypersurfaces in {$\mathbb
  S^{n+1}$} by gluing spherical building blocks}, Math. Z. \textbf{263} (2009),
  no.~1, 1--25. \MR{2529485}

\bibitem{B2}
\bysame, \emph{Gluing constructions amongst constant mean curvature
  hypersurfaces in {$\mathbb S^{n+1}$}}, Ann. Global Anal. Geom. \textbf{36}
  (2009), no.~3, 221--274. \MR{2544302}

\bibitem{BP}
Adrian Butscher and Frank Pacard, \emph{Generalized doubling constructions for
  constant mean curvature hypersurfaces in {$S^{n+1}$}}, Ann. Global Anal.
  Geom. \textbf{32} (2007), no.~2, 103--123. \MR{2336180}

\bibitem{Carron01}
G.~Carron, \emph{{{\(L^ 2\)}}-harmonic forms on non-compact manifolds}, Rend.
  Mat. Appl., VII. Ser. \textbf{21} (2001), no.~1-4, 87--119 (French).

\bibitem{COR}
Marcos~P. Cavalcante, Darlan~F. de~Oliveira, and Robson dos~S. Silva,
  \emph{Index bounds for closed minimal surfaces in 3-manifolds with the
  killing property}, Mat. Contemp. \textbf{50} (2022), 38--53. \MR{4505852}

\bibitem{CO}
Marcos~Petr\'{u}cio Cavalcante and Darlan~Ferreira de~Oliveira, \emph{Lower
  bounds for the index of compact constant mean curvature surfaces in {$\mathbb
  R^3$} and {$\mathbb S^3$}}, Rev. Mat. Iberoam. \textbf{36} (2020), no.~1,
  195--206. \MR{4061986}

\bibitem{Chern}
Shiing-shen Chern, \emph{On surfaces of constant mean curvature in a
  three-dimensional space of constant curvature}, Geometric dynamics, {Proc}.
  int. {Symp}., {Rio} de {Janeiro}/{Brasil} 1981, {Lect}. {Notes} {Math}. 1007,
  104-108 (1983)., 1983.

\bibitem{Daniel}
Beno{\^{\i}}t Daniel, \emph{The {Gauss} map of minimal surfaces in the
  {Heisenberg} group}, Int. Math. Res. Not. \textbf{2011} (2011), no.~3,
  674--695.

\bibitem{DFM}
Beno{\^{\i}}t Daniel, Isabel Fern{\'a}ndez, and Pablo Mira, \emph{The {Gauss}
  map of surfaces in {{\(\widetilde{\mathrm {PSL}}_2(\mathbb R)\)}}}, Calc.
  Var. Partial Differ. Equ. \textbf{52} (2015), no.~3-4, 507--528.

\bibitem{EFFR}
N.~do~Esp{\'{\i}}rito-Santo, S.~Fornari, K.~Frensel, and J.~Ripoll,
  \emph{Constant mean curvature hypersurfaces in a {Lie} group with a
  bi-invariant metric}, Manuscr. Math. \textbf{111} (2003), no.~4, 459--470.

\bibitem{Ejiri_Micallef}
Norio Ejiri and Mario Micallef, \emph{Comparison between second variation of
  area and second variation of energy of a minimal surface}, Adv. Calc. Var.
  \textbf{1} (2008), no.~3, 223--239.

\bibitem{FM}
Isabel Fern{\'a}ndez and Pablo Mira, \emph{Harmonic maps and constant mean
  curvature surfaces in {{\(\mathbb H^2\times\mathbb R\)}}}, Am. J. Math.
  \textbf{129} (2007), no.~4, 1145--1181.

\bibitem{folha}
Abigail Folha and Carlos Pe{\~n}afiel, \emph{The {Gauss} map and second
  fundamental form of surfaces in a {Lie} group}, Ann. Mat. Pura Appl. (4)
  \textbf{195} (2016), no.~5, 1693--1711.

\bibitem{GalvezMira}
Jos{\'e}~A. G{\'a}lvez and Pablo Mira, \emph{A {Hopf} theorem for non-constant
  mean curvature and a conjecture of {A}. {D}. {Alexandrov}}, Math. Ann.
  \textbf{366} (2016), no.~3-4, 909--928.

\bibitem{Ka1}
Nikolaos Kapouleas, \emph{Complete constant mean curvature surfaces in
  {E}uclidean three-space}, Ann. of Math. (2) \textbf{131} (1990), no.~2,
  239--330. \MR{1043269}

\bibitem{Ka2}
\bysame, \emph{Compact constant mean curvature surfaces in {E}uclidean
  three-space}, J. Differential Geom. \textbf{33} (1991), no.~3, 683--715.
  \MR{1100207}

\bibitem{Ka3}
\bysame, \emph{Constant mean curvature surfaces constructed by fusing {W}ente
  tori}, Invent. Math. \textbf{119} (1995), no.~3, 443--518. \MR{1317648}

\bibitem{Leung}
Pui~Fai Leung, \emph{On the stability of harmonic maps}, Harmonic maps ({N}ew
  {O}rleans, {L}a., 1980), Lecture Notes in Math., vol. 949, Springer,
  Berlin-New York, 1982, pp.~122--129. \MR{673586}

\bibitem{Masaltsev}
L.~A. Masal'tsev, \emph{A version of the {R}uh-{V}ilms theorem for surfaces of
  constant mean curvature in {$S^3$}}, Mat. Zametki \textbf{73} (2003), no.~1,
  92--105. \MR{1993542}

\bibitem{Mazet}
E.~Mazet, \emph{La formule de la variation seconde de l'energie au voisinage
  d'une application harmonique}, J. Differ. Geom. \textbf{8} (1973), 279--296
  (French).

\bibitem{Palmer91}
Bennett Palmer, \emph{Index and stability of harmonic {G}auss maps}, Math. Z.
  \textbf{206} (1991), no.~4, 563--566. \MR{1100840}

\bibitem{Palmer93}
\bysame, \emph{Stability of harmonic {G}auss maps}, Geometry and topology of
  submanifolds, {V} ({L}euven/{B}russels, 1992), World Sci. Publ., River Edge,
  NJ, 1993, pp.~233--241. \MR{1339977}

\bibitem{RamosRipoll}
\'{A}lvaro Ramos and Jaime Ripoll, \emph{An extension of {R}uh-{V}ilms' theorem
  to hypersurfaces in symmetric spaces and some applications}, Trans. Amer.
  Math. Soc. \textbf{368} (2016), no.~7, 4731--4749. \MR{3456159}

\bibitem{ripoll91}
Jaime~B. Ripoll, \emph{On hypersurfaces of {L}ie groups}, Illinois J. Math.
  \textbf{35} (1991), no.~1, 47--55. \MR{1076665}

\bibitem{Ruh}
Ernst~A. Ruh, \emph{Asymptotic behaviour of non-parametric minimal
  hypersurfaces}, J. Differential Geometry \textbf{4} (1970), 509--513.
  \MR{276877}

\bibitem{RV}
Ernst~A. Ruh and Jaak Vilms, \emph{The tension field of the {G}auss map},
  Trans. Amer. Math. Soc. \textbf{149} (1970), 569--573. \MR{259768}

\bibitem{Savo}
Alessandro Savo, \emph{Index bounds for minimal hypersurfaces of the sphere},
  Indiana Univ. Math. J. \textbf{59} (2010), no.~3, 823--837. \MR{2779062}

\bibitem{SeemungalSharp}
Luca Seemungal and Ben Sharp, \emph{Index estimates for constant mean curvature
  surfaces in three-manifolds by energy comparison}, Preprint,
  {arXiv}:2411.02932 [math.{DG}] (2024), 2024.

\bibitem{Smith}
R.~T. Smith, \emph{The second variation formula for harmonic mappings}, Proc.
  Amer. Math. Soc. \textbf{47} (1975), 229--236. \MR{375386}

\bibitem{wente}
Henry~C. Wente, \emph{Counterexample to a conjecture of {H}. {H}opf}, Pacific
  J. Math. \textbf{121} (1986), no.~1, 193--243. \MR{815044}

\end{thebibliography}
 

\end{document}